\documentclass[11pt,a4paper]{article}

\usepackage{cite}
\usepackage{amsmath,bm}
\usepackage{amsfonts}
\usepackage{amssymb}
\usepackage{extarrows}
\usepackage{latexsym}
\usepackage[T1]{fontenc}

\usepackage{xcolor}

\def\N{\mathbb{N}}
\def\R{\mathbb{R}}

\def\d{\mbox{d}}

\def\x{{\mathbf x}}
\def\xu{{\textup x}}
\def\t{{\textup t}}

\def\ul#1{{\mathbf #1}}
\def\bs#1{{\bm{#1}}}
\def\wh#1{{\widehat{\mathbf #1}}}
\def\wha#1{{\widehat{#1}}}
\newcommand{\widebar}[2][3]{{}\mkern#1mu\overline{\mkern-#1mu#2}}
\def\ol#1{{\widebar{#1}}}
\def\up{\ul u^\prime}

\newcommand{\SC}{\scriptstyle}

\newcommand{\vc}[1]{{\mathbf #1}}
\newcommand{\vv}{\vc{v}}
\newcommand{\vr}{\rho}
\newcommand{\vt}{\theta}
\newcommand{\vm}{\vc{m}}

\newcommand{\abs}[1]{{\left| #1 \right|}}

\newcommand{\RE}[2]{R_E\left( #1 \mid #2 \right)}
\newcommand{\MI}{\widehat{\mathcal{I}}}
\newcommand{\MT}{\mathcal{I}}
\newcommand{\TA}{\mathcal{T}_A}

\numberwithin{equation}{section}

\newtheorem{thm}{Theorem}[section]

\newenvironment{proof} {{\bf Proof: }}{\hfill $\blacksquare$ \\[2mm]}

\begin{document}


\title {A note on relations between convexity and concavity of thermodynamic functions}

\author{
M\'aria Luk\'a\v{c}ov\'a-Medvid'ov\'a\thanks{\baselineskip=.5\baselineskip Institut f\"ur Mathematik, Universit\"at Mainz, 
Staudingerweg 9, 55128 Mainz, Germany.  Email: lukacova@uni-mainz.de} $\;\;$
Ferdinand Thein\thanks{ \baselineskip=.5\baselineskip Institut f\"ur Mathematik, Universit\"at Mainz, 
Staudingerweg 9, 55128 Mainz, Germany.  Email: fthein@uni-mainz.de} $\;\;$ 
Gerald Warnecke\thanks{ \baselineskip=.5\baselineskip Institute of Analysis and  Numerics,
Otto-von-Guericke-University Magdeburg, PSF 4120, D--39016 Magdeburg,
Germany.  Email: warnecke@ovgu.de} $\;\;$
Yuhuan Yuan\thanks{ \baselineskip=.5\baselineskip  School of Mathematics, Nanjing University of Aeronautics and Astronautics, 
Jiangjun Avenue No. 29, 211106 Nanjing, P. R. China. Email: yuhuanyuan@nuaa.edu.cn} 
}

\date {\today }

\maketitle

\begin{abstract}
The paper is concerned with proving the equivalence of convexity or concavity properties of thermodynamic
functions, such as energy and entropy, depending on
different sets of variables. These variables are the basic thermodynamic state variables, specific state variables or
the densities of state variables that are used in continuum mechanics.  We prove results for transformations of
variables and functions in conjunction with convexity properties. We are concerned with convexity, strict convexity, 
positive definite Hessian matrices and the analogous forms of concavity. The main results are equivalence relations
for these properties between functions. These equivalences are independent of the 
equations of state since they only use general properties of them.
The results can be used for instance to easily prove that the
entropy density function for the Euler equations in conservative variables in
three space dimensions is strictly concave or even has a negative definite Hessian matrix. 
Further,
we show how various equations of
state imply these properties and how these properties are relevant to mathematical analysis. 
\end{abstract}


\section{Introduction}

Concavity properties of entropy as a function of other state variables are important for a number reasons. 
They are linked to thermodynamic
stability, the use of the Legendre transformation in thermodynamics and the hyperbolicity of systems of conservation laws 
\eqref{conslaw}. By the latter the initial value problem
for the system is well posed. There is a theory of local in time existence of smooth solutions to these systems.
Also concavity properties play a decisive role in convergence analysis of numerical schemes for systems of conservation laws.

Thermodynamic stability theory implies that entropy is a concave function of state variables, see Callen \cite[Chapter 8]{bCAL}
or Kondepudi and Prigogine \cite[Chapter 5]{bKOPR}. Actually, strict concavity is needed if a unique maximum state is to
be attained. Also, strict concavity is sufficient for the Legendre transformations used in thermodynamics to be bijective
transformations.

Boillat \cite{hBOI} proved that a system of conservation laws being hyperbolic is equivalent 
to the entropy being negative concave, i.e.\ having a negative
definite Hessian matrix. So this property is closely linked to the system being hyperbolic and
having well posed initial value problems.

Positive symmetric systems of equations were introduced by Friedrichs \cite{hFRI1} for quite general linear systems,
including systems of mixed type. The theory was
extended to a local in time existence theory of solutions for sufficiently smooth initial data to quasilinear hyperbolic systems 
by Kato \cite{hKAT}. Later accounts of this theory can be found in Majda \cite[Chapter 2]{bMAJ}, Serre \cite[Section 3.6]{bSER3}, 
Alinhac and Gerard \cite[Section III.B]{bALGE}, 
Benzoni-Gavage and Serre \cite[Section 10.1]{bBESE} or Liu \cite[Section 6.4]{bLIU}. 
These results do not include solutions that contain discontinuities such as shock waves or contact discontinuities.

Godunov \cite{hGOD1} showed how the system of Euler equations of gas dynamics can be turned into a positive symmetric hyperbolic
system of Friedrichs type using an entropy with everywhere negative definite Hessian matrices. 
It turns into an entropy function with positive definite Hessians by taking a minus sign.
The theory was put into a general framework of 
symmetric hyperbolic thermodynamically compatible (SHTC) systems of conservation laws and applied to numerous systems
by Romenski \cite{hROM1,hROM2}. Then it was explored together with Toro \cite{hROTO} and other co-authors \cite{hRRT,hRDT,hRBP,hDPR}.
See also the literature cited in the references just mentioned. Based on papers by Godunov \cite{hGOD1}, Friedrichs and Lax 
\cite{hFRLA} and Boillat \cite{hBOI}, a theory of rational extended thermodynamics was developed by
M\"uller and Ruggeri \cite{bMURU1}, see also Ruggeri and Sugiyama \cite{bRUSU1,bRUSU2}.

The convexity of the energy and the concavity of the entropy provide the structural framework for the analysis of 
partial differential equation (PDE) models of compressible fluid flows. We refer, for example, to the fundamental works on the 
local-in-time existence of a strong solution to the Euler equations of gas dynamics of Majda \cite{bMAJ} 
and the global-in-time existence of 
weak or dissipative weak solutions of the Navier-Stokes-Fourier system of viscous compressible heat-conductive fluid flows 
of Feireisl et al.\ \cite{bFLMS, bFENO}, where the convex structures play a crucial role.  
The thermodynamic stability and the resulting 
convexity of the energy yield the relative energy functional that acts as the Bregman distance and can be used for the 
comparison of two different solutions, see e.g.\ Feireisl et al.\ \cite{hBRFE, bFENO} and 
Luk\'a\v cov\'a-Medvi\v dov\'a et al.\ \cite{hnBEHL,  hnLISH, hnLSY, hnLSY1}.
A more detailed description of the application to thermodynamic stability in mathematical and numerical analysis 
of fluid flows will be presented in Section \ref{sec:role}. 
 
All these considerations need at least the strict concavity of physical entropy, some additionally the negative concavity, i.e.\
that the Hessian matrix of entropy as a function of some variables is negative definite. The main point of this paper is 
to provide transformations of variables and functions that preserve these properties. This shows how properties of convexity and
concavity of various thermodynamic functions are related.

The main new results of the paper concern two transformations of variables and functions that maintain
convexity and concavity properties of thermodynamic functions. These are the reciprocal involution and the interchange of 
a variable with a function. The former was introduced as general convexity preserving transformation
by Godunov and Romenskii \cite[\S 20]{bGORO}. The latter
has been commonly used in the field of hyperbolic conservation laws, see e.g.\ Friedrichs and Lax \cite{hFRLA},
Liu and Thein \cite{hLITH}.
They can also be found in Croisille \cite{hCRO}. 
Godunov and Romenskii as well as Croisille gave proofs that show the preservation of convexity and
strict convexity under these transformations. We extend their results to cases in which convexity is transformed to concavity.
New are also our proofs that definiteness properties of the Hessian matrices are transformed as well.
Strict convexity does not imply that the Hessian of a function is positive definite. Therefore,
definiteness properties have to be proven separately.
For completeness, well known results for the Legendre transformation are recalled.
We point out that the reciprocal involution allows the transformation between specific state variables and their densities
when the specific volume is used.
It seems that this has not been widely recognized before. 

Using these transformations one may start with simple thermodynamic stability criteria internal energy or
entropy and obtain convexity or concavity properties of various other thermodynamic functions.
The results on the transformations lead to a short and elegant proof of the negative concavity of the entropy
density $\ol S(\rho,\ol{\ul M},\ol E)$ as a function of the conserved variables of the system of Euler 
equations under appropriate assumptions. These variables are the densities of mass $\rho$, momentum $\ol{\ul M}$
and total energy $\ol E$.
The negative concavity means that the Hessian matrix is negative definite.
An equivalent result of positive definite Hessian matrices is shown for the total energy as a function of the
mass, momentum and entropy densities.
Our proof does not assume 
a specific equation of state, but only general conditions for it.
This fact is frequently used in the literature. But there are very few references to proofs.
We are aware of only two references to a complete proof in the case of three space
dimensions for a general equation of state. One was given by Godunov and Peshkov \cite{hGOPE2}. The second
is in the recent edition of Godlewski and Raviart \cite[Section III.1]{bGORA1}. Our results generalize
steps that they used.

The one dimensional case for a polytropic equation of state is quite straight forward to prove.
Liu and Thein \cite[Proposition 1]{hLITH} proved for general equations of state the negative definite Hessian matrix
in one space dimension.
Harten \cite{hHAR} gave a proof for the two dimensional case and the polytropic equation of state using Sylvester's criterion
for the positive definiteness of the Hessian of the entropy density with a minus sign $-\ol S$. 
He showed that the leading principal minors are all positive. 
In the three dimensional case this approach to the proof would require to show that the 
determinant of a $5\times 5$ fully occupied matrix 
is positive, involving $5! =120$ nontrivial terms.

Luk\'{a}\v{c}ov\'{a}-Medvid'ov\'{a} and Yuan \cite{hLUYU} gave a proof of lower bounds of the eigenvalues of the Hessian of $-\ol S$
via the characteristic polynomial.
The proof is very technical and partially relies on computer algebra. Feireisl et al.\ \cite[Subsection 2.2.3]{bFLMS} 
gave a proof of the concavity of a larger class of entropy densities $\ol S(\rho,p)$ infering the concavity of
$\ol S(\rho,\ol{\ul M},\ol E)$. 
The convexity of the total energy density $\ol E(\rho,\ol{\ul M},\ol S)$ was shown
in \cite[Section 4.1.6]{bFLMS} using the relative energy.
In Section \ref{sec:appl} we show how these results can be proven directly by simple chains of transforming Hessian matrices.

In Godunov and Romenski \cite[Subsection 20.6]{bGORO} there is an interesting approach to a proof that the generating
potential is convex or strictly convex in the main field variables. 
By the Legendre transformation this implies convexity of $-S$
in the conserved variables. We can extend this to the case of positive definite Hessian matrices.
As we point out at the end of Section \ref{sec:appl}, this line of argument allows one to
generate symmetric hyperbolic systems for conserved densities from simple assumptions on the internal energy of
a system, see the comments at the end of Section \ref{sec:appl}.

There is a rich textbook literature of convex analysis and optimization with well defined notions of convexity and
strict convexity.
But unfortunately, there has been a somewhat loose use of notions of convexity in the community of hyperbolic conservation laws.
Also there is some confusion concerning definiteness properties of the Hessian matrices in conjunction with
strict convexity or concavity of smooth functions. In some instances it seems to be suggested that 
strict convexity implies a positive definite
Hessian matrix. This is wrong, see e.g.\ Ortega and Rheinboldt \cite[Section 3.4]{bORRH}. 
The function $x^4$ is a simple counterexample. 
Convexity of a smooth function is indeed equivalent to positive semi-definite Hessian.
For example Friedrichs and Lax \cite{hFRLA} as well as
Harten \cite[p.\ 153]{hHAR} used the term
convexity as being equivalent to a positive definite Hessian matrix.
Analogously, Liu and Thein \cite[(3.1)]{hLITH} termed the case of a negative definite Hessian as concave.
Harten et al.\ \cite[p.\ 2120]{hHLLM} as well as Godlewski and Raviart \cite[Section III.1]{bGORA1}
defined strict convexity to mean the case of a positive definite Hessian matrix.
In their proof of a theorem Berger and Berger \cite[Theorem 4-1]{bBEBE} make the wrong statement of equivalence
of strict convexity and a positive definite Hessian matrix, but use only the correct implication in their proof.
Another example is an incorrect statement of equivalence in Warnecke \cite[Lemma 6.7]{bWAR}.

The paper is structured as follows.
In the next section we discuss convex and concave functions. We consider various transformations of such
functions that preserve these properties or switch between the two. These are the Legendre transformation, the reciprocal involution,
an exchange of variables with functions and affine transformations of variables. Then in Section
\ref{sec:cons} we give a short outline of Godunov's theory for conservation laws and symmetric hyperbolic systems.
In Section \ref{sec:thermo}
we recall some notions of thermodynamics that are fundamental to the application of the convexity results
of Section \ref{sec:convex}. This includes the thermodynamic stability criteria.
As examples we show in Section \ref{sec:conv_ther} how for the ideal gas and the Tait equation of state the stability criteria
are satisfied. Then in Section \ref{sec:appl} we apply the results from the second section
to various thermodynamic functions. Some of the relations are not new, but we are not aware of any presentation
of the complete picture that we give. Finally, in the last section we illustrate the relevance of these properties in 
mathematical analysis of systems of partial differential equations and of numerical schemes for compressible fluid flows.

\section{Convex functions and involutions}
\label{sec:convex}

Of special interest in this paper are properties of convex or concave functions as well as convexity preserving
transformations of variables and functions. Some of these transformations
are {\bf involutions} in the following sense: If you
repeat the same transformation procedure to the new variables and functions you recover the old ones, 
i.e.\ the procedure is self-inverse.
The Legendre transformation has been known for a long time to be an involution that preserves convexity.
Godunov and Romenskii \cite[\S 20]{bGORO} identified two more involutions that preserve convexity. We will consider them below.
The first is a reciprocal involution that can be used to go from the specific quantities to densities and back. Specific
quantities and densities will be introduced in 
Section \ref{sec:thermo}. The other transformation is an exchange of variables and functions. Examples
will be considered in Subsection \ref{sec:pot_state}. 

After recalling some properties of convex functions we will introduce the involutions of the system. We give results on
their convexity preserving properties. If the convex function has a positive definite Hessian matrix, we will also consider
whether the involutions maintain the positive definite Hessian when transforming a convex function. Further, we consider the conditions
under which convexity properties are transformed to concavity properties. This is the case when one goes from energy to entropy.

\subsection{Convex and concave functions, monotone operators}

Let $D \subseteq \R^m$ be a convex open set. We will always assume all sets not to be empty.
We define that a function $\phi: D \to \R $ is 
{\bf convex} if and only if for any $\ul u,\ul v \in D$ the inequality
\begin{equation} 
\label{eq:convex}
\phi\left( s \ul u + (1-s) \ul v \right) \le s \phi(\ul u) +   (1-s)\phi(\ul v)
\end{equation}
holds for all $s \in ]0,1[$. If moreover the strict inequality $<$ holds for any $\ul u\ne\ul v$, the function $\phi$ is 
{\bf strictly convex}. Further, the function $-\phi$ is then {\bf concave} 
respectively {\bf strictly concave}. For a concave function the inequality in \eqref{eq:convex}
is reversed. 

For differentiable $\phi$ there are equivalent conditions of convexity. 
We denote the scalar product between vectors $\ul u,\ul v\in\R^m$ 
by $\ul u\cdot\ul v$. One condition is the inequality
\begin{equation}
\label{eq:conv_ineq}  
\phi(\ul u)-\phi(\ul v) \geq \bs \phi_{\ul v}(\ul v)\cdot (\ul u-\ul v)
\end{equation}
holds for all $\ul u,\ul v \in D$, see Ortega and Rheinboldt \cite[3.4.4.]{bORRH}. In the case of strict convexity
it is a strict inequality for all $\ul u,\ul v \in D$, $\ul u\ne\ul v$.
We also make use of the the inequality
\begin{equation}
\label{eq:conv_ineq1} 
[\bs \phi_{\ul u}(\ul u) - \bs \phi_{\ul v}(\ul v)] \cdot  (\ul u-\ul v)> 0 
\end{equation}
for all $\ul u,\ul v \in D$, $\ul u\ne\ul v$. It is equivalent to the
strict convexity of $\phi$, see \cite[3.4.5.]{bORRH}. 
For the Legendre transformation, see Subsection \ref{sec:legendre}, we will want to know that the change of 
variables $\ul w =\bs\phi_{\ul u} (\ul u)$ under the {\bf gradient map} $\bs\phi_{\ul u}$  is injective. This is 
immediately implied the above inequality for strictly convex functions.
This simple result can be found in Berger and Berger \cite[\S 4-1]{bBEBE}. One can make the gradient map
bijective by restricting the image set.

Now suppose that $\phi:D\to\R^m$ is twice continuously differentiable on the open and convex set $D\subset\R^n$.
Then we can use the Hessian matrices $\bs \phi_{\ul u\ul u}(\ul u)$ for all $\ul u\in D$ to characterize
convexity or concavity of $\phi$. We call a function $\phi$ {\bf positive convex} when its Hessian matrices
are positive definite and {\bf negative concave} in the negative definite case. The well known criteria for positive definiteness apply
to Hessian matrices $\bs \phi_{\ul u\ul u}(\ul u)$ point wise for each $\ul u$. We call the Hessian matrices
{\bf uniformly positive definite} if and only if there is a constant $c>0$ such that
$\bs \phi_{\ul u\ul u}(\ul u)\ul x\cdot\ul x \ge c\ul x\cdot\ul x$ for
all $\ul x\in \R^m$ and all $\ul u\in D$. We then call $\phi$ {\bf uniformly positive convex}.
Suppose that the Hessian matrices $\bs \phi_{\ul u\ul u}(\ul u)$ are point wise positive definite on $D$ and
continuous functions of $\ul u$. Then they
are uniformly positive definite on any compact subset $K\subset D$. The same is true for the inverse
matrices $\bs \phi_{\ul u\ul u}^{-1}(\ul u)$.

Suppose that the function $\phi:D\to\R$ is twice differentiable on the open and convex set $D\subset\R^n$.
The function $\phi$ is convex if and only if the Hessian matrix $\bs \phi_{\ul u\ul u}(\ul u)$
is positive semi-definite, i.e.\ $\bs \phi_{\ul u\ul u}(\ul u) \ul w\cdot\ul w \ge 0$ 
for all $\ul u\in D$ and all $\ul w\in\R^m$.
If we assume that the Hessian matrix $\bs \phi_{\ul u\ul u}(\ul u)$ is positive definite,
i.e.\ $\bs \phi_{\ul u\ul u} (\ul u)\ul w\cdot\ul w > 0$ for all $\ul u\in D$ and all $\ul w\in\R^m\backslash\{\ul 0\}$,
then function is strictly convex. 
The converse is not true, i.e.\ strict convexity only implies that the Hessian matrix is positive semi-definite.
The real function $\phi (u)=u^4$ is an example that is strictly convex but has 
a vanishing second derivative at $u=0$.
These results can e.g.\ be found in Ortega and Rheinboldt \cite[Section 3.4]{bORRH}.

\subsection{The Legendre transformation}
\label{sec:legendre}

Let us assume that $\phi:D\to \R$ is a smooth strictly convex function on the convex set $D\subset\R^m$. We
set $\widehat{D}=\bs\phi_{\ul u} [D]$ to be the image set of the gradient mapping. As noted above,
the gradient map $\MT_L =\bs\phi_{\ul u}:D\to\widehat{D}$ is a bijective
transformation of variables to $\ul w=\MT_L(\ul u)\in\widehat{D}$. Thereby it has 
the inverse $\ul u =\MT_L^{-1}\ul w$.
This transformation of variables is the {\bf Legendre transformation} associated to $\phi$.

Its importance comes from the well known fact that it is induced by the map $\MI_L$ of the strictly convex functions
$\phi:D\to \R$ to strictly convex functions $\psi:\widehat{D}\to \R$  with $\psi = \MI_L(\phi)$
given as
\begin{equation}
\label{eq:legendre}
\psi(\ul w) =\ul w\cdot\ul u - \phi(\ul u)
\end{equation}
for any $\ul w=\MT_L(\ul u)\in\widehat{D}$. Differentiating \eqref{eq:legendre} with respect to $\ul u$
gives transformation of variables $\ul w =\bs\phi_{\ul u}(\ul u)=\MT_L(\ul u)$.
Putting the defining equation  \eqref{eq:legendre} into the form
$\phi(\ul u) + \psi(\ul w) -\ul u\cdot \ul w =0$, we see that the two functions $\phi$ and $\psi$ are related in a dual manner. 
If we switch the variables $\ul u$ and $\ul w$,
we have that $\phi$ is also the Legendre transform of $\psi$. 
The Legendre transform induces, in a slightly generalized sense, an {\bf involution}, i.e.\ applying the
same procedure twice gives the identity mapping.
We may disregard the fact that the functions differ in their variables since the variables are identical modulo the bijective 
transformation $\MT_L$.

Differentiating \eqref{eq:legendre} with respect to $\ul w$ gives $\psi_{\ul w}(\ul w)=\ul u$. So we have $\MT^{-1}_L =\psi_{\ul w}$.
By further differentiation we obtain
$\MT_{L\ul u}(\ul u) =\bs \phi_{\ul u\ul u}(\ul u)$ and 
$\MT^{-1}_{L\ul w}(\ul w) =\bs \psi_{\ul w\ul w}(\ul w)$ 
for the Hessian matrices of $\phi$ and $\psi$. We consider
the identity $\ul w = \MT_L(\MT_L^{-1}(\ul w))$ and differentiate with respect to $\ul w$.
On the left hand side we obtain the indentity matrix $\ul I$ as Jacobian matrix.
The chain rule gives
$$
\ul I = \MT_{L\ul u}(\MT_L^{-1}(\ul w))\MT^{-1}_{L\ul w}(\ul w)=\MT_{L\ul u}(\ul u)\MT^{-1}_{L\ul w}(\ul w)=
\bs \phi_{\ul u\ul u}(\ul u)\bs \psi_{\ul w\ul w}(\ul w).
$$
 This implies that $\bs \psi_{\ul w\ul w}=\bs \phi_{\ul u\ul u}^{-1}$.
Both symmetric matrices have the same definiteness properties.

Suppose that we set $\psi(\ul w) = \phi(\ul u)-\ul w\cdot\ul u$ instead of \eqref{eq:legendre}. 
Then we obtain an analogous transformation that
turns convexity into concavity and vice versa. This type of transformation appears in the thermodynamic definitions of enthalpy
and free energies, see M\"uller and M\"uller \cite[(4.33)]{bMUMU}.

\subsection{The reciprocal involution}

Let us assume that either $u_1>0$ or $u_1<0$ for all $\ul u \in D$ and define the transformation of variables
\begin{equation}
\label{eq:recip}
\ul w =\MT_R(\ul u) =\left(\textstyle\frac 1{u_1}, \frac{u_2}{u_1},\ldots, \frac{u_m}{u_1}\right)^T.
\end{equation}
We set $\widehat{D}=\MT_R[D]$. The map $\MT_R$ is an involution and bijection since 
$\MT_R^{-1}(\ul w) =(\frac 1{w_1}, \frac{w_2}{w_1},\ldots, \frac{w_m}{w_1})^T$ is the same mapping.
We call $\MT_R$ the {\bf reciprocal involution} of variables associated to $u_1$.
Further, we can define a mapping $\MI_R$ of functions $\phi:D\to\R$ to functions $\psi:\widehat{D}\to R$ by
\begin{equation}
\label{recip_fct}
\psi(\ul w) =\MI_R(\phi)(\ul w)
=w_1\phi\left(\textstyle\frac 1{w_1}, \frac{w_2}{w_1},\ldots, \frac{w_m}{w_1}\right)=w_1\phi (\MT_R^{-1}(\ul w)).
\end{equation}
Then \eqref{recip_fct} gives
$$
\textstyle
\phi(\MT_R^{-1}(\ul w))=\phi\left(\frac 1{w_1}, \frac{w_2}{w_1},\ldots, \frac{w_m}{w_1}\right)=\frac 1{w_1}\psi(\ul w) 
=u_1\psi\left(\frac 1{u_1}, \frac{u_2}{u_1},\ldots, \frac{u_m}{u_1}\right)=u_1\psi(\ul w).
$$
So this transformation of the functions is also an involution. Note that this transformation does not depend on the choice of 
the functions $\phi$ or $\psi$. We introduce the notations $\wh u =(u_2,\ldots,u_m)^T\in\R^{m-1}$ and
$\wh w =(w_2,\ldots,w_m)^T\in\R^{m-1}$.
\begin{thm}[Reciprocal involution and convexity]
\label{thm:recinv}
Let us assume that $D\subset\R_{>0}\times\R^{m-1}$ is open, convex and that $\phi:D\to \R$ is continuously differentiable. 
Let $\phi$ and $\psi$ satisfy \eqref{recip_fct}. Then $\psi$ is continuously differentiable on any
open subset of $\widehat{D} =\MT_R[D]\subset\R_{>0}\times\R^{m-1}$.
If one of the functions $\phi$ and $\psi$ is convex, strictly convex, 
concave or strictly concave, then so is the other function.

Let us assume that $\phi:D\to \R$ is twice continuously differentiable with respect to $\ul u$ and that 
the Hessian matrix $\bs \phi_{\ul u\ul u}$
is positive definite or uniformly positive definite. Then the Hessian matrix $\bs \psi_{\ul w\ul w}$
is positive definite, respectively uniformly positive definite.

If instead $D\subset\R_{<0}\times\R^{m-1}$, i.e.\ $u_1<0$, then we can use the same
transformation but the properties are switched, e.g.\ if $\phi$ is convex then $\psi$ is concave.
\end{thm}
\begin{proof}
In Godunov and Romenskii \cite[\S 20.6]{bGORO} one can find a proof of the convexity or strict convexity of $\psi$.

Our assumption $D\subset\R_{>0}\times\R^{m-1}$ implies that $u_1>0$. 
We use the shorthand notation $\ul u(\ul w) =\MT_R^{-1}(\ul w)$.
Let us calculate the first derivatives of $\psi$. We have
\begin{eqnarray}
\label{psi_w1}
\psi_{w_1} (\ul w) &=& \phi(\ul u(\ul w)) -w_1\left[ \phi_{u_1}(\ul u(\ul w))\left(\frac 1{w_1}\right)^2
+ \bs \phi_{\wh u}(\ul u(\ul w))\cdot\wh w\left(\frac 1{w_1}\right)^2\right]\nonumber\\
&=& \phi(\ul u(\ul w)) -\frac 1{w_1}\left[ \phi_{u_1}(\ul u(\ul w))
+ \bs \phi_{\wh u}(\ul u(\ul w))\cdot\wh w\right]\\
&=&\phi(\ul u(\ul w)) -\left[ \phi_{u_1}(\ul u(\ul w))u_1
+ \bs \phi_{\wh u}(\ul u(\ul w))\cdot\wh u\right]\nonumber\\
&=& \phi(\ul u(\ul w)) - \bs\phi_{\ul u}(\ul u(\ul w))\ul u\nonumber
\end{eqnarray}
and 
$$
\psi_{\wh w} (\ul w) =w_1\phi_{\wh u}(\ul u(\ul w))\frac 1{w_1}=\phi_{\wh u}(\ul u(\ul w)).
$$

Now we calculate the second derivatives. We use the second equation in \eqref{psi_w1} and have
\begin{eqnarray*}
\psi_{w_1w_1} (\ul w) &=& -\left(\frac 1{w_1}\right)^2 \big(\phi_{u_1}(\ul u(\ul w))+ \bs\phi_{\wh u}(\ul u(\ul w))\cdot\wh w\big)
+\left(\frac 1{w_1}\right)^2 \big(\phi_{u_1}(\ul u(\ul w))+ \bs\phi_{\wh u}(\ul u(\ul w)\cdot\wh w\big)\\
&&+\left(\frac 1{w_1}\right)^3\Big[\phi_{u_1u_1}(\ul u(\ul w))
+2\bs\phi_{u_1\wh u}(\ul u(\ul w))\cdot\wh w +(\bs\phi_{\wh u\wh u}(\ul u(\ul w))\wh w)\cdot\wh w\Big]\\
&=&u_1\Big[(u_1)^2\phi_{u_1u_1}(\ul u(\ul w))
+2u_1\bs\phi_{u_1\wh u}(\ul u(\ul w))\cdot\wh u+(\bs\phi_{\wh u\wh u}(\ul u(\ul w))\wh u)\cdot\wh u\Big]\\
&=& u_1(\bs\phi_{\ul u\ul u}(\ul u(\ul w))\ul u)\cdot\ul u,
\end{eqnarray*}
\begin{eqnarray*}
\bs\psi_{w_1\wh w} (\ul w)&=&\bs\psi_{\wh ww_1} (\ul w)=
-\left(\frac 1{w_1}\right)^2(\bs\phi_{u_1\wh u}(\ul u(\ul w))+\bs\phi_{\wh u\wh u} (\ul u(\ul w))\cdot\wh w)\\
&=&-u_1[u_1\bs\phi_{u_1\wh u}(\ul u(\ul w))+\bs\phi_{\wh u\wh u} (\ul u(\ul w))\cdot\wh u)]
\end{eqnarray*}
and 
$$
\bs\psi_{\wh w\wh w} (\ul w)=\bs\phi_{\wh u\wh u} (\ul u(\ul w))\frac 1{w_1} =u_1\bs\phi_{\wh u\wh u} (\ul u(\ul w)) .
$$
In case $u_1>0$ we have $\psi_{w_1w_1}(\ul w)>0$.
We obtain the symmetric block matrix
$$
\bs\psi_{\ul w\ul w} =u_1{\scriptstyle\left(\begin{array}{c|c} 
(\bs\phi_{\ul u\ul u}(\ul u(\ul w))\ul u)\cdot\ul u
& -[u_1\bs\phi_{u_1\wh u}(\ul u(\ul w))^T+(\bs\phi_{\wh u\wh u}(\ul u(\ul w))\cdot\wh u)^T]\\
\hline
-[u_1\bs\phi_{u_1\wh u}(\ul u(\ul w))&\\
+\bs\phi_{\wh u\wh u}(\ul u(\ul w))\cdot\wh u]&\bs\phi_{\wh u\wh u}(\ul u(\ul w))
\end{array}\right)}.
$$
We use the notation for the identity matrix $\ul I\in\R^{(m-1)\times(m-1)}$ 
and for the vector $\ul 0 =(0,\ldots ,0)^T\in\R^{m-1}$.
With these we introduce the Frobenius matrix $\ul F\in\R^{m\times m}$ for Gaussian type elimination
and the diagonal matrix $\ul D\in\R^{m\times m}$ as
$$
\ul F =\left(
\begin{array}{c|c}
\textstyle 1 & \ul 0^T\\[1mm]
\hline
\wh u &\ul I
\end{array}\right),  \qquad
\ul D =\left(
\begin{array}{c|c}
\textstyle -\frac{1}{u_1} & \ul 0^T\\[1mm]
\hline
\ul 0 &\ul I
\end{array}\right).
$$
These eliminate some terms in the first row and column to give
$$
\ul D^T\ul F^T\bs\psi_{\ul w\ul w}\ul F\ul D=
u_1\ul D{\scriptstyle\left(\begin{array}{c|c} 
\SC (u_1)^2\phi_{\ul u_1\ul u_1}(\ul u(\ul w))
&\SC -u_1\bs\phi_{u_1\wh u}(\ul u(\ul w))^T\\
\hline
&\\
\SC -u_1\bs\phi_{u_1\wh u}(\ul u(\ul w))&\SC\bs\phi_{\wh u\wh u}(\ul u(\ul w))\\
&
\end{array}\right)}\ul D =u_1\bs\phi_{\ul u\ul u}.
$$
With the sign of $u_1$ the claimed results follow.
\end{proof}
\subsection{Involution by exchange of a variable with a function}

A function $\varphi:D\to\R$ with $D\subset\R^m$ for some $m\in\N$ is {\bf monotone increasing}, respectively {\bf decreasing},
in some variable $u_i$ for $1\le i\le n$ if for any $\ul u_1,\ul u_2\in\Omega$ with $u_{i,1}< u_{i,2}$, 
while any other components of $\ul u_1$ and $\ul u_2$ are equal, implies that $\varphi(\ul u_1)\le\varphi(\ul u_2)$, 
respectively $\varphi(\ul u_1)\ge\varphi(\ul u_2)$.
The function is {\bf strictly monotone increasing}, respectively {\bf decreasing}, in this variable
if the strict inequality $\varphi(\ul u_1)<\varphi(\ul u_2)$, respectively $\varphi(\ul u_1)>\varphi(\ul u_2)$,
holds. 

The condition $\varphi_{u_i}(\ul u)\ge 0$, respectively $\varphi_{u_i}(\ul u)\le 0$, for all
$\ul u\in D$ is equivalent to $\varphi$ being monotone increasing, respectively decreasing in that variable.
We write $\varphi_{u_i}\ge 0$, respectively $\varphi_{u_i}\le 0$, as a brief notation.
The conditions $\varphi_{u_i}> 0$, respectively $\varphi_{u_i}< 0$, imply strict monotonicity.
The reverse is not true, since for example $\varphi(u) =u^3$ is strictly monotone increasing but $\varphi^\prime(u)=0$
at $u=0$. We will see an analogue of this non-equivalence for the positive definiteness of Hessian matrices
and strict convexity.

Let $\phi:D\to \R$ be a continuously differentiable convex function on the open convex set $D\subset\R^{m}$
with $\phi_{u_1}(\ul u)<0$ for all $\ul u\in D$. Then $\phi$ is a strictly monotone decreasing function of $u_1$. Again we set
$\wh u =(u_2,\ldots ,u_m)^T\in\R^{m-1}$ for $\ul u=(u_1,\ldots ,u_m)^T\in D$. For each such $\wh u$ we have a global bijection between
$u_1$ and $\phi(\ul u)$. Therefore, we may
choose $\phi$ as a new variable to replace $u_1$. The variables $\wh u$ are parameters in this transformation. For each value of
$\phi$ the dependence of $u_1$ on $\wh u$ is guaranteed by the implicit function theorem. Let $\ul P_1:D\to\R$ be the mapping
to the first variable, i.e.\ $\ul P_1\ul u =u_1$.
We introduce the new variables $\ul w$ and and the {\bf exchange of variables involutions} $\MT_E:D\to\R^m$ 
and $\MI_E$ via
$$
\ul w= \MT_E(\ul u) = (\phi(\ul u),u_2,\ldots, u_m)^T\in\R^m,\qquad\psi(\ul w)=\MI_E(\phi)(\ul w) = \ul P_1\MT_E^{-1}(\ul w)
=\ul P_1\ul u =u_1
$$
for $\ul w\in\widehat{D}=\MT_E[D]$ with 
parameters $\wh w =(w_2,\ldots ,w_m)^T=\wh u\in\R^{m-1}$. Note that this transformation does depend on the choice of $\phi$ ore $\psi$,
as in the case of the Legendre transformation.
\begin{thm}[Exchanging variables and functions]
\label{thm:e}
Let us assume that $\phi:D\to \R$ is continuously differentiable with respect to $\ul u$ on the open convex set $D$,
$\phi_{u_1}(\ul u)<0$ for all $\ul u\in D$ and $\ul w = (\phi,u_2,\ldots, u_m)^T\in\R^m$. 
Let $\psi = \MI_E(\phi)$.  Then $\psi$ is continuously differentiable on any
open subset of $\widehat{D} =\MT_E[D]\subset\R^m$.
If one of the functions $\phi$ and $\psi$ is convex, strictly convex, 
concave or strictly concave, then so is the other function. 

Let us assume that $\phi:D\to \R$ is twice continuously differentiable with respect to $\ul u$ and that 
the Hessian matrix $\bs \phi_{\ul u\ul u}$
is positive definite or semi-definite. Then the Hessian matrix $\bs \psi_{\ul w\ul w}$
is positive definite, respectively positive semi-definite. The same is true for negative definite or
semi-definite Hessian matrices.

If $\phi_{u_1}(\ul u)>0$ for all $\ul u\in D$, then we can use the same
transformation but the properties are switched, e.g.\ if $\phi$ is convex then $\psi$ is concave.
\end{thm}
\begin{proof}
In Godunov and Romenskii \cite[\S 20.6]{bGORO} one can find a proof of the convexity or strict convexity of $\psi$.

The first derivatives of $\psi$ are 
$$
\psi_{w_1}(\ul w) =\frac 1{\phi_{u_1}(\psi(\ul w),\wh w)}\qquad\text{and}\qquad
\psi_{w_i}(\ul w) =-\frac {\phi_{u_i}(\psi(\ul w),\wh w)}{\phi_{u_1}(\psi(\ul w),\wh w)}\quad\text{for}\; i =2,\ldots m.
$$
Now we calculate the second derivatives of $\psi$ 
$$
\psi_{w_1w_1} =-\frac 1{(\phi_{u_1})^3}\phi_{u_1u_1},\qquad \psi_{w_1w_i} =\psi_{w_iw_1} 
=\frac {\phi_{u_i}}{(\phi_{u_1})^3}\phi_{u_1u_1}
-\frac 1{(\phi_{u_1})^2}\phi_{u_1u_i}\quad\text{for}\; i=2,\ldots,m
$$
and
$$
\psi_{w_iw_j} =-\frac {\phi_{u_i}\phi_{u_j}}{(\phi_{u_1})^3}\phi_{u_1u_1}+\frac {\phi_{u_i}}{(\phi_{u_1})^2}\phi_{u_1u_j}
+\frac {\phi_{u_j}}{(\phi_{u_1})^2}\phi_{u_1u_i}-\frac 1{\phi_{u_1}}\phi_{u_iu_j}
\quad\text{for}\; i,j=2,\ldots,m.
$$
We define the following Frobenius matrix $\ul F\in\R^{m\times m}$ and diagonal matrix $\ul D\in\R^{m\times m}$
using the identity matrix
$\ul I\in\R^{(m-1)\times(m-1)}$ as well as the vector $\ul 0 =(0,\ldots ,0)^T\in\R^{m-1}$ as
$$
\ul F =\left(
\begin{array}{c|c}
\textstyle 1&\bs\phi_{\wh u}^T\\[1mm]
\hline
\ul 0 &\ul I
\end{array}\right), \qquad
\ul D =\left(
\begin{array}{c|c}
\textstyle\phi_{u_1} & \ul 0^T\\[1mm]
\hline
\ul 0 &\ul I
\end{array}\right).
$$
We obtain with $i,j= 2,\ldots, m$ again by Gaussian type elimination
$$
\ul D^T\ul F^T\bs\psi_{\ul w\ul w}\ul F\ul D =-\frac 1{\phi_{u_1}}\bs\phi_{\ul u\ul u}.
$$
By the sign of $\frac 1{\phi_{u_1}}$ the results follow.
\end{proof}
\subsection{Affine transformations}

Affine transformations of coordinates preserve convexity.
Let $D\subset\R^m$ be a convex set and $\ul \phi:D\to\R$ be a convex, strictly convex, concave or strictly concave 
function. We consider any affine function $\TA:\wha D\to D$ with the convex set 
$\wha D =\{\ul w\in\R^n | \TA\ul w\in D\}\ne\emptyset$ 
that is given by some matrix $\ul T\in\R^{m\times n}$ and some vector
$\ul b\in\R^m$ as $\TA\ul w =\ul T\ul w+\ul b$. Associated to the transformation of variables via $\TA$ we have 
the transformation of functions $\widehat{\mathcal{T}}_A$ defined via
$$
\psi =\widehat{\mathcal{T}}_A(\phi)=\phi\circ\TA:\R^n\to \R
$$ 
with $\psi (\ul w) =\phi(\TA\ul w) =\phi (\ul u)$
for $\ul u=\TA\ul w\in D$ and $\ul w\in\wha D$. It has the respective property of $\phi$, as can easily be seen by
e.g.\ using \eqref{eq:convex}. 
These transformations do not depend on $\phi$ or $\psi$.

Suppose that the set $D$ is open as well as convex, that $\phi$ is twice differentiable 
with a positive or negative definite or semi-definite Hessian $\bs\phi_{\ul u\ul u}$
and that $\ul T$ is a regular, i.e.\ invertible, matrix.
Then $\psi$ is twice differentiable with a Hessian $\bs\psi_{\ul w\ul w}$ having the same property as $\bs\phi_{\ul u\ul u}$.
If $\ul T$ is not regular, then $\bs\psi_{\ul w\ul w}$ is only semi-definite, even if $\bs\phi_{\ul u\ul u}$ is definite.
By the chain rule we have $\bs\psi_{\ul w\ul w} =\ul T^T\bs\phi_{\ul u\ul u}\ul T$. This yields the desired properties.

A later useful special case is that $\ul T$ is a diagonal matrix with diagonal entries $\pm 1$ and $\ul b=\ul 0$. This amounts to
sign changes in the variables, i.e.\ the functions 
\begin{equation}
\label{eq:sign}
\phi(\pm u_1,\ldots, \pm u_m)
\end{equation}
inherit all properties of $\phi$ and its Hessian matrix if it is a smooth function.

\section{Conservation laws and symmetric hyperbolic systems}
\label{sec:cons}

Let us denote the spatial variables as $\x=(\xu_1,\ldots ,\xu_d)\in\R^d$, for $d=1,2,3$, $|\cdot|$ 
the Euclidean norm in $\R^d$ and the time variable $\t\in \R_{\geq 0}$. We consider an appropriate convex 
subset of physically admissible states
$D\subseteq \R^m$ with non-empty interior.
Then we consider for the column vector 
field of conserved variables $\ul u =(u_1,\ldots ,u_m)^T\in D$ the system of $m\in \N$ conservation laws in divergence form
\begin{equation}
\label{conslaw}
\ul u_\t + \sum_{i=1}^d \ul f^i ( \ul u) _{\xu_i} = \ul 0 
\end{equation}
with smooth flux functions $\ul f^i:D\to \R^m$.
Denote their Jacobian matrices as $\ul A^i =\ul f^i_{\ul u}=(\partial_{u_k}f^i_j)_{1\le j,k\le m}$. 
The system \eqref{conslaw} is called {\bf hyperbolic} if for any $\ul u\in D$ all eigenvalues $\lambda(\ul u,\ul n)$ of the 
matrices $\ul A(\ul u,\ul n)=\sum_{i=1}^d n_i\ul A^i(\ul u)$,
for any unit vector $\ul n\in\R^d$ with $|\ul n|=1$, are real and there is complete set of
$m$ linearly independent eigenvectors $\ul r(\ul u,\ul n)\in\R^m$. 

\subsection{Entropy functions and symmetrization}

Many systems \eqref{conslaw} have for $\ul u\in D$ a convex function $\Phi:D\to \R$ which is also a conserved quantity.
It is called the {\bf entropy function} because for some systems it is related to the 
concave thermodynamic entropy density $S=-\Phi$, see Section \ref{sec:thermo}. 
As a matter of taste mathematicians prefer to work with convex functions instead of concave functions.
They differ only by a minus sign.
Such an entropy function must come together with 
{\bf entropy fluxes} $ \Psi^i:D\to\R$ for $i=1,\ldots ,d$ to give the additional conservation law
\begin{equation}
\label{entcons}
\Phi(\ul u)_\t + \sum_{i=1}^d \Psi^i(\ul u)_{\xu_i} =0.
\end{equation}
The physical entropy is a concave function on state space. 
The formulations differ only by a minus sign since negating a convex function makes it concave and vice versa, 
see Section \ref{sec:convex}. 

For a scalar function like $\Phi$ we denote by $\Phi_{\ul u}$ the gradient vector with respect to $\ul u$ and 
by $\Phi_{\ul u\ul u}$ the Hessian matrix. 
All smooth solutions to the system \eqref{conslaw} satisfy also the conservation law \eqref{entcons}
for an entropy function $\Phi:D\to\R$ with associated flux functions $ \Psi^i:D\to\R$ for $i=1,\ldots ,d$ iff the relations
$(\Psi^i_{\ul u})^T=\Phi_{\ul u}^T\ul f^i_{\ul u}=\Phi_{\ul u}^T\ul A^i$ hold for $i=1,\ldots,d$.

The entropy function $\Phi$ can be used to transform the system \eqref{conslaw} into a positive symmetric system. 
This was first discovered by Godunov \cite{hGOD1,hGOD2}.
Later Friedrichs and Lax \cite{hFRLA} found a second possibility of symmetrization that is less advantageous since it
does not maintain the jump relations for weak solutions. 
The vector of the main field or entropy variables is 
given as $\up(\ul u)=\bs\Phi_{\ul u}(\ul u)$. It was first introduced by Godunov \cite{hGOD1}. 
Ruggeri and Strumia \cite{hRUST} coined the name main field.
Any system \eqref{conslaw} with a strictly convex entropy function and corresponding entropy fluxes satisfying
the consistency conditions $(\Psi^i_{\ul u})^T=\Phi_{\ul u}^T\ul f^i_{\ul u}$ can be transformed into such variables. 

Let us consider the Legendre transform of the entropy function $\Phi$ and an analogous transformation
of its fluxes 
\begin{equation}
\label{legendre}
L(\up) = \up\cdot\ul u -\Phi(\ul u),\qquad
L^i(\up) = \up\cdot\ul f^i(\ul u) -\Psi^i(\ul u)
\end{equation}
with $\up =\bs\Phi_{\ul u}(\ul u)$. If we assume that $\Phi$ is strictly convex this transformation
of variables is bijective to $\ul u[D]$ via \eqref{eq:conv_ineq1}. 
The new functions $L$, $L^i$ defined by the transformations are
{\bf generating potentials}.  
These functions are all potentials in the
usual meaning that their gradients $L_{\up}(\up ) =\ul u$ and $L^i_{\up}(\up ) =\ul f^i(\ul u)$ are 
physically important vector fields.

We can insert $\ul u(\up)$ into $\ul L_{\up}(\up(\ul u)) =\ul u$ 
and $\ul L^i_{\up}(\up(\ul u)) =\ul f^i(\ul u)$.
Then the system \eqref{conslaw} can be rewritten in the main field variables 
$\up$ to give
\begin{equation}
\label{entvar-law2}
\ul L_{\up}(\up)_\t + \sum_{i=1}^d \ul L^i_{\up}(\up)_{\xu_i}=\ul u(\up)_\t + \sum_{i=1}^d\ul f^i(\ul u(\up))_{\xu_i} =\ul 0.
\end{equation}
The conserved quantities $\ul u$ remain the same. They are only expressed in different variables. 
The quasilinear form of \eqref{entvar-law2} is
\begin{equation}
\label{symm2}
\ul u_{\up}(\up)\up_\t + \sum_{i=1}^d\ul f^i_{\ul u}(\ul u(\up))\ul u_{\up}(\up)\up_{\xu_i}
=\ul L_{\up\up}(\up)\up_\t + \sum_{i=1}^d \ul L^i_{\up\up}(\up)\up_{\xu_i}=\ul 0.
\end{equation}
It is a symmetric system because the coefficient matrices 
$\ul L_{\up\up}(\up)$ and $\ul L^i_{\up\up}(\up)$ 
for $i=1,\ldots ,d$ are Hessian matrices and hence symmetric. 
If we want to use results from the theory of positive symmetric hyperbolic systems,
we have to require that the Hessian matrix $\ul L_{\up\up}(\up)$ is positive definite for all $\ul u$ of interest.
For a more detailed survey of this theory see Warnecke \cite{hWAR1}.

\section{Thermodynamic states and their relations}
\label{sec:thermo}

Basic state variables of any substance that are quantifiable by measurement are the volume $V$ it is contained in, 
the pressure $p$ it is subject to and its absolute temperature $\theta$.
These state variables are called the {\bf thermal state variables}, see e.g.\ Schmidt \cite[Section 3.2]{bSCHM}, because their value
can be determined by measurement.
There are relations between the thermal state variables called thermal equations of state. 
They are determined from many measurements and embody the particular properties of the material substance 
that is being considered in an application.
 
The caloric state variables are that state variables that cannot be
measured directly. But they are very useful for the theory of thermodynamics, see Schmidt \cite[Section 3.2]{bSCHM}.
Their values have to be determined from thermal state variables using formulas. 
Important caloric state variables are the internal energy $U$, 
and the entropy $S$ of the substance. 
The internal energy $U$ is determined from the thermal state variables as a function of $V$ and $\theta$
called {\bf caloric equation of state}, see M\"uller and M\"uller \cite[Prologue P.1]{bMUMU} 
or Schmidt \cite[Chapter 12]{bSCHM}. 
The former reference uses the mass density instead of volume.
Actually, in some instances, such as an ideal gas, $U$ is only a function of the temperature $\theta$, 
see Kondepudi and Prigogine \cite[Section 1.4]{bKOPR}, M\"uller and M\"uller \cite[Subsection 1.5.9]{bMUMU}
or Schmidt \cite[Section 12.1]{bSCHM}.

The state variables volume $V$, 
internal energy $U$, 
and entropy $S$ are {\bf extensive state variables} depending on the size of the system under consideration,
i.e.\ their value doubles when the system is doubled in size. Another measurable 
extensive state variable is the mass $M$ of a substance.
{\bf Intensive state variables} are characterized by the fact that their value remains the same when the size of the system is changed. 
This is true for the pressure $p$ and the absolute temperature $\theta$. 

\subsection{Specific state variables}

The term {\bf specific} is used for a physical state variable taken per unit mass, i.e.\ it is divided by the mass $M$. 
They belong to the category of intensive state variables.  
We use small letters for specific thermodynamic state variables as well as velocity $\ul v$ as a specific 
dynamic state variable. 
Important specific state variables are the specific thermal state variable specific volume $v=\frac VM$ and
the specific caloric state variables specific internal 
energy $u=\frac UM$ as well as specific entropy $s=\frac SM$. Further, we have the
specific total energy $e = u + \frac{|\ul v|^2}{2}$ being the sum of
specific internal energy $u$ and the specific kinetic energy $e^{kin}=\frac{|\ul v|^2}{2}$.

\subsection{Densities}

The term {\bf density} in connection with a physical state variable means that the state variable
is being taken per unit volume, i.e.\ it is divided by the volume $V$.
For some extensive state variables
that have associated densities we use a bar over the same capital letter to denote the density. For instance
$\ol {S}$ for the entropy density. 
An exception is the mass density denoted as $\rho =\frac MV$.

Densities of extensive state variables are obtained from specific state variables by
multiplying the latter by the mass density $\rho$, e.g.\ the specific internal energy $u$ becomes the 
internal energy density $\ol{U}=\rho u$ or the velocity vector $\ul v$ then becomes 
the momentum density vector $\ol{\ul M}=\rho\ul v$. 
In this terminology the velocity $\ul v$ could also be called the specific momentum.

We will also need the total energy $E=U+M\frac{|\ul v|^2}2$.
It is the sum of the internal energy $U$
and the kinetic energy $E^{kin}=M\frac{|\ul v|^2}2$. 
They also correspond to the internal energy density $\ol{U} =\rho u$,
the kinetic energy density $\ol{E}^{kin}=\rho e^{kin}=\frac{|\ol{\ul M}|^2}{2\rho}$ and the total energy density 
$$
\ol{E}(\rho,u,\ul v)=\rho e(u,\ul v) = \ol{U}(\rho,u) + \ol{E}^{kin}(\rho,\ul v)
=\rho u +\rho\frac{|\ul v|^2}{2}\qquad\text{or}\qquad
\ol{E}(\rho ,\ol{U},\ol{\ul M}) = \ol{U}+\frac{|\ol{\ul M}|^2}{2\rho}.
$$
\subsection{Fundamental equations and generating potentials}
\label{sec:pot_state}

In the following text we will frequently make changes of state variables in functions that map to another state variable. 
Since we would run out of notations we do not rename the functions with the new variables. We always use the image
state variable to denote the function. We assume that all thermodynamic functions, that we take into consideration,
are continuously differentiable as often as we need, at least twice. 

Suppose that we are given the extensive state variable internal energy $U$ as a 
function of the extensive state variables entropy $S$ and volume $V$, i.e.\ we are considering $U(V,S)$.
This function is a generating potential for other state variables. When $U$ is given as a function of
these variables the resulting equation is called the {\bf fundamental equation} of thermodynamics, see Callen \cite[Section 2.1]{bCAL}.
In principle the internal energy $U$ can be expressed as a function of any two of the variables $S$, $V$, $\theta$
or $p$ as its arguments. 

We will use the notation $U_S$ for $\frac{\partial U}{\partial S}$ etc.
From the first law of thermodynamics and the second law for reversible processes one can derive
the equivalent equations 
\begin{equation}
\label{eq:gibbs}
\d U=\theta\,\d S -p\,\d V \qquad\text{or}\quad
\d S  =\frac 1\theta \big(\d U+p\,\d V\big),
\end{equation}
see Callen \cite[(2.6)]{bCAL} or Kondepudi and Prigogine \cite[(4.1.2)]{bKOPR}.
Variants of these are sometimes called {\bf Gibbs equation}, see e.g.\ M\"uller and M\"uller \cite[(4.19),(4.20),(8.9)]{bMUMU}.  
 
The Gibbs equation \eqref{eq:gibbs} for the internal energy $U$ is equivalent to the relations
\begin{equation}
\label{eq:theta_p}
U_V=\left(\frac{\partial U}{\partial V}\right)_{S}=-p\qquad\text{and}\qquad 
U_S=\left(\frac{\partial U}{\partial S}\right)_{V}=\theta.
\end{equation}
So $U$ is a potential for the vector field $\left(\theta , -p\right)^T$. The values of
the components of this vector field may be determined via \eqref{eq:theta_p} if $U(V,S)$ is known.
Another consequence of \eqref{eq:gibbs} uses the equality of mixed second derivatives such as $U_{SV}=U_{VS}$.
The resulting equations, such as $\theta_V =- p_S$, are called {\bf Maxwell relations}.

For the entropy $S$ the Gibbs equation \eqref{eq:gibbs} gives analogouly
\begin{equation}
\label{eq:S_deriv}
S_V=\left(\frac{\partial S}{\partial V}\right)_{U}=\frac p\theta \qquad\text{and}\qquad
\quad S_U=\left(\frac{\partial S}{\partial U}\right)_{V}=\frac 1\theta .
\end{equation}
So $S$ is a potential for the vector field 
$\left(\frac 1\theta , \frac p\theta\right)^T$. 
The values of the components of this vector field may be determined if $S(V,U)$ is known.
We see that the theory is slightly more elegant if $U$ is used as a potential instead of $S$.

The surprising result of Godunov \cite{hGOD1}
was that there are analogously a potential for the vector of conserved variables 
and the flux vectors of the system of Euler equations of gas dynamics.
Godunov \cite{hGOD1} mentioned that his search for a potential function $L$ given by
\eqref{legendre} and leading to the equations \eqref{entvar-law2} was motivated by the potentials of thermodynamics. 

We generally assume that the absolute temperature and the pressure satisfy 
$\theta ,p>0$. This means that $U_V,U_S>0$.
We can use $U_S>0$ and Theorem \ref{thm:e} to replace $U(V,S)$ by $S(V,U)$ as our fundamental equation or potential.
Then the second
Gibbs equation \eqref{eq:gibbs} $\d S = \frac 1\theta \left(\d U+p \,\d V\right)$ holds.

\section{Convex or concave thermodynamic functions}
\label{sec:conv_ther}

There are some physical restrictions on the nature of the potentials that we now discuss. Kondepudi and Prigogine
\cite[Chapter 5]{bKOPR} discuss a number of extremal principles for various potentials, see also Callen \cite[Chapter 8]{bCAL}.
This topic is also called {\bf thermodynamic stability}. Thermodynamic processes will maximize entropy or minimize
some form of energy depending on the nature of the process.

The second law of thermodynamics states that in an isolated system entropy is maximized in a thermodynamic process.
For the extremal state to be a unique maximum the entropy has to be strictly concave.
The Hessian matrix of the entropy $S$ must therefore be be at least negative semi-definite, as pointed out in Section \ref{sec:convex}.
This is equivalent to
\begin{equation}
\label{eq:stab}
S_{VV}\le 0,\qquad S_{UU}\le 0\qquad\text{and}\qquad S_{VV}S_{UU}-S_{VU}^2\ge 0.
\end{equation}
These conditions are called {\bf thermodynamic stability conditions}, see Callen \cite[Chapter 8]{bCAL}.

Above we showed how $U(V,S)$ is obtained from $S(V,U)$. Theorem \ref{thm:e} then implies that 
$U$ is a strictly convex function of the variables $V$ and $S$ satisfying
\begin{equation}
\label{eq:u_posdef}
U_{VV}\ge 0,\qquad U_{SS}\ge 0\qquad\text{and}\qquad U_{VV}U_{SS}-U_{VS}^2\ge 0.
\end{equation}
Of the two non-mixed derivative conditions only one is actually required. The non-negative determinants in \eqref{eq:stab}
and \eqref{eq:u_posdef} imply that these derivatives must have the same sign. So only two stability conditions are needed.
These conditions and \eqref{eq:theta_p} give $U_{VV}=-p_V\ge 0$, $U_{SS}=\theta_S\ge 0$,
$U_{VS}=-p_S=\theta_V$ and
$p_V\theta_S-\theta_Vp_S=-U_{VV}U_{SS}+U_{VS}^2\le 0$.

The conditions \eqref{eq:u_posdef}, using only one non-mixed derivative, are Sylvester's criterion 
for a positive semi-definite $2\times 2$ matrix,
see e.g.\ Strang \cite[Section 6.3]{bSTR}. The conditions \eqref{eq:stab} are for a negative semi-definite $2\times 2$ matrix.
They are obtained noting that $-S$ is then positive semi-definite.

For some purposes of mathematical analysis one may need that the Hessian matrix of entropy is negative definite,
i.e.\ that $S$ is negative concave.
Then the inequalities in \eqref{eq:stab} and \eqref{eq:u_posdef} have to be strict. 
We define this case as {\bf strong thermodynamic stability}. 
This will carry over to all transformations of variables considered in Section \ref{sec:convex} when applied to such functions.

\subsection{Stability conditions in other variables}

For practical purposes it is advantageous to have stability conditions that are directly related
to measurable quantities.
From \eqref{eq:theta_p} we have $\theta(V,S) = U_S(V,S)$. We assume $\theta$ to be positive. 
Let us further assume strong stability, i.e.\ the strict inequalities in 
the stability conditions \eqref{eq:u_posdef}. Then we have $\theta_S=U_{SS}>0$. This implies that $\theta$ is a
strictly monotone increasing function of $S$ and we can invert this function to $S(V,\theta)$. This form of the entropy
has the advantage that it is a function of measurable quantities. In the case of strong stability or
positive convexity of $U$, it is strictly monotone increasing
in $\theta$ with $S_\theta(V,\theta) = \frac 1{U_{SS}(V,S(V,\theta))}>0$. With the new variables for the entropy
$S$, we obtain the caloric equation of state $U(V,\theta)=U(V,S(V,\theta))$.
From \eqref{eq:theta_p} we have $\theta(V,S) = U_S(V,S(V,\theta))$. With this we obtain
$$
U_\theta(V,\theta) = U_S(V,S(V,\theta))S_\theta(V,\theta)=\theta(V,S)\frac 1{U_{SS}(V,S(V,\theta))}>0.
$$
So we have $U_\theta(V,\theta)>0$ as an equivalent stability condition 
in the variables $V$ and $\theta$ replacing the condition $U_{SS}(V,S)>0$. 
Using the results in Subsection \ref{subsec:ext_int} below
we can analogously obtain $u_\theta(v,\theta)>0$. The quantity $c_v =u_\theta(v,\theta)$
is the specific heat capacity at constant volume. It is a quantity that can be measured,
see M\"uller and M\"uller \cite[(2.16),(2.20)]{bMUMU}.
Further, we may substitute $v=\frac 1\rho$ to give $u_\theta(\rho,\theta)>0$.

We have the relation $\theta=\theta(V,S(V,\theta))$ where the left hand side is a variable. 
Differentiating with respect to $V$ gives
$$
0=\theta_V(V,S(V,\theta))+\theta_S(V,S(V,\theta))S_V(V,\theta)\qquad\text{or}\qquad S_V
=-\frac{\theta_V(V,S)}{\theta_S(V,S)}=-\frac{U_{SV}(V,S)}{U_{SS}(V,S)}.
$$
From \eqref{eq:theta_p} we also have $p(V,S) = -U_V(V,S)$. This leads for the
thermal equation of state $p(V,\theta) = -U_V(V,S(V,\theta))$ to
\begin{align*}
p_V(V,\theta) &=-U_{VV}(V,S(V,\theta))-U_{VS}(V,S(V,\theta))S_V(V,\theta)\\
&=-\left(U_{VV}(V,S(V,\theta))-U_{VS}(V,S(V,\theta))\frac{U_{VS}(V,S(V,\theta))}{U_{SS}(V,S(V,\theta))}\right)<0,
\end{align*}
due to the strict inequality in the determinant condition from \eqref{eq:u_posdef}. So we have found $p_V(V,\theta)<0$
as the second condition for strong stability using the variables $V$ and $\theta$. Again we can analogously proceed
to obtain $p_v(v,\theta)<0$.
Taking $\rho$ instead of $v$ reverses the sign in the derivative to yield $p_\rho(\rho,\theta)>0$. So strong stability is equivalent
$u_\theta(\rho,\theta)>0$ and $p_\rho(\rho,\theta)>0$. These conditions were used for instance in Feireisl et al.\ \cite{bFLMS}
in conjunction with mathematical analysis for the Euler equations of gas dynamics.

The rate of change in pressure with respect to volume $p_V$ at constant temperature 
and thereby $p_v$ can be determined by measurement. Due to $p_v(v,\theta)<0$ one may also
invert this function to give $v(p,\theta)$. 
The quantity $\kappa_\theta =-\frac 1v v_p(p,\theta)=-\frac 1{vp_v(v,\theta)}$ is the compressibility
of the material, see e.g.\ M\"uller and M\"uller \cite[(2.12)]{bMUMU}.

Note that all steps above are reversible.
So one can determine strong stability of the fundamental equations for $U$ and $S$, i.e.\
strict inequalities in \eqref{eq:stab} and \eqref{eq:u_posdef},
from the conditions
\begin{equation}
\label{eq:up}
u_\theta(v,\theta)>0\;\;\text{or}\;\;u_\theta(\rho,\theta)>0\qquad\text{and}\qquad p_v(v,\theta)<0\;\;\text{or}\;\;
p_\rho(\rho,\theta)>0
\end{equation}
on the caloric and thermal equations of state respectively.

\subsection{The ideal polytropic gas and the van der Waals equation of state}

Now we want to consider an ideal polytropic gas as a specific equation of state and its compatibility with the convexity conditions such
as for example \eqref{eq:stab} or \eqref{eq:u_posdef}. We will see in Subsection \ref{subsec:ext_int}
that these conditions are equivalent for extensive and the corresponding intensive variables.

For the description of an ideal polytropic gas we need the 
{\bf specific heat capacities} at constant volume $c_v=u_\theta(v,\theta)$ 
and at constant pressure $c_p$ They are measurable material
properties that we assume to be constant. With these we define the {\bf specific gas constant} ${\cal R} =c_p-c_v$
It is quite common to use the {\bf adiabatic exponent} 
$$
\gamma := \frac{{\cal R}}{c_v} +1  =  \frac{c_p}{c_v}>1.
$$
Details may be found e.g.\ in M\"uller and M\"uller \cite[(2.26),(4.36)]{bMUMU} where they also point out that $c_p>c_v$ and thereby
$\gamma >1$ as well as ${\cal R}>0$ hold for all substances. Typical values are for a mono atomic gas $\gamma = 5/3=1.66\ldots$,
for a diatomic gas $\gamma = 7/5=1.4$ and for a higher atomic multiplicity $\gamma =4/3=1.33\ldots$,
see M\"uller and M\"uller \cite[1.5.8]{bMUMU}. The adiabatic exponent $\gamma$ has no units. Due to the two given
relations between them, only two of the values of ${\cal R}$, $c_v$, $c_p$ or $\gamma$ have to be known.
Then the other two are determined.

The {\bf ideal gas law} is a thermal equation of state given as
\begin{equation}
\label{ideal}
p(\rho,\theta) = {\cal R} \rho \theta .
\end{equation}
In a polytropic gas, see Courant and Friedrichs \cite{bCOFR}, we assume also the caloric equation of state
\begin{equation}
\label{calor}
u(\theta) = c_v \, \theta .
\end{equation}
For the ideal gas it does not depend on the mass density $\rho$, see M\"uller and M\"uller \cite[Prologue P.1]{bMUMU}.
We see that the strong stability conditions \eqref{eq:up} hold iff $c_v>0$.

Together the ideal gas law \eqref{ideal} and \eqref{calor} give the well known {\bf polytropic equation of state}
\begin{equation}
\label{poly}
p(\rho ,u) =\frac {{\cal R}}{c_v}\rho u=(\gamma -1)\rho u.
\end{equation}

The {\bf van der Waals} thermal equation of state replaces the ideal gas law by
$$
p(v,\theta) =\frac {{\cal R}\theta}{v-b}-\frac a{v^2},
$$
see M\"uller and M\"uller \cite[(2.33)]{bMUMU}. 
It has two constants $a,b>0$ that are adjusted for practical purposes. It is defined for $v>b$.
We have from \eqref{eq:up} the stability condition 
$$
p_v(v,\theta) =\frac{-{\cal R}\theta}{(v-b)^2}+\frac{2a}{v^3}<0.
$$
The inequality holds for $v$ being slightly larger that $b$ and for large $v$. For lower
temperatures $\theta$ it is violated and there is an interval of physically inadmissible states $v$ related to
a phase change. The van der Waals caloric equation of state is
$$
u(v,\theta) = -\frac av+ F(\theta)
$$
with $F^\prime(\theta)=c_v(\theta)$, see M\"uller and M\"uller \cite[(4.26)]{bMUMU}.
So the stability condition $u_\theta(v,\theta)>0$ from \eqref{eq:up} corresponds to $c_v(\theta)>0$.
This condition is always satisfied.

\subsection{The Tait equation of state}

We additionally want to consider the {\bf Tait equation of state} as a further example which fulfills the convexity requirements. 
This was shown in Liu and Thein \cite{hLITH}. There a detailed presentation was given. 
For additional literature concerning the Tait equation of state we refer for example to 
Dymond and Malhotra \cite{hDYMA}, Ivings et al.\cite{hICT}, Saurel et al.\ \cite{hSCB}, Hantke et al.\ \cite{hHDW1} 
or May and Thein \cite{hMATH}.

	Let the specific internal energy be given as follows
	\begin{equation}
	\label{eq:e_eos}
		u(v,s) = A(v - v_r) + B\Phi(v) + \frac{1}{2C}\left[(s - s_r) - D(v - v_r)\right]^2
		+ \theta_r\left[(s - s_r) - D(v - v_r) + C\theta_r\right] + u_r
		\end{equation}
		with
		$$
		\Phi(v) = 
		\begin{cases}
			&\ln\dfrac{v_r}{v},\qquad\text{for}\;\nu = 1\\[2mm]
			&\dfrac{1}{1 - \nu}\left(\dfrac{1}{v_r^{\nu - 1}} - \dfrac{1}{v^{\nu - 1}}\right),\;\text{for}\;\nu > 1
		\end{cases}\,.
	$$
	The exponent $\nu \geq 1$ can be chosen according to the specific material under consideration.
	The constants $A,B,C,D$ will be specified later.
	The quantities $u_r,v_r, s_r$ and $\theta_r$ denote reference constants of the quantities which may be 
	chosen at a given reference state in the phase space of the material at hand.
	 
	Now we calculate the pressure and the temperature according to the formulas \eqref{eq:theta_p} as
	\begin{align}
	\label{eq:p_tait}
		p &= -\frac{\partial u}{\partial v}(v,s) = -\left(A + B\Phi'(v) - 
		\frac{D}{C}\left[(s - s_r) - D(v - v_r)\right] - D\theta_r\right),\\
		\label{eq:T_tait}
		\theta &= \frac{\partial u}{\partial s}(v,s) = \frac{1}{C}\left[(s - s_r) - D(v - v_r)\right] + \theta_r.
	\end{align}
	Using equation (\ref{eq:T_tait}) for the temperature we obtain the thermal equation of state
	\begin{align}
		p(v,\theta) = D(\theta - \theta_r) - A - B\Phi'(v).\label{eq:p_tait_vT}
	\end{align}
	We choose
	\[
	A := K_r - p_r,\quad B := K_rv_r^\nu\quad\text{and}\quad\bar{p}(\theta):= D(\theta - \theta_r) + p_r.
	\]
	The constant quantity $K_r > 0$ denotes the modulus of compression at a given reference state.
	The constants $K_r,v_r$ may be chosen as the saturation values of the material under consideration for a 
	given reference temperature $\theta_r$, cf.\ Thein \cite{hTHE}.
	Accordingly, $\bar{p}(\theta)$ is obtained as the linearization of the saturation curve near a given temperature. 
	A similar approach is used in Hantke and Thein \cite{hHATH1}.
	Usually, the saturation curve is given as the saturation pressure being a function of the temperature
	\[
	p_0 = f(\theta),
	\]
	see e.g.\ Landau and Lifschitz \cite{bLALI80} or Thein \cite{hTHE}, and thus we would obtain	the linearization
	\[
	\bar{p}(\theta) = f'(\theta_r)(\theta - \theta_r) + p_r.
	\]
	With these assumptions we can rewrite (\ref{eq:p_tait}) into a more well known form
	\begin{align}
	\label{eq:p_tait_vT2}
		p = \bar{p}(\theta) + K_r\left[\left(\frac{v_r}{v}\right)^\nu - 1\right].
	\end{align}
	The requirement $p > 0$ places a restriction on the admissible region of the state space.
	In order to verify the convexity of the internal energy we calculate the second order derivatives
	\begin{align*}
		\frac{\partial^2 u}{\partial v^2}(v,s) = B\Phi''(V) + \frac{D^2}{C},\quad
		\frac{\partial^2 u}{\partial s^2}(v,s) = \frac{1}{C},\quad
		\frac{\partial^2 u}{\partial s\partial v}(v,s) = -\frac{D}{C}.
	\end{align*}
	It is easily verified that $u(v,s)$ is strictly convex if $C > 0$. Indeed $C$ is related to the specific isochoric heat capacity
	\[
	c_v = \frac{\partial u}{\partial \theta}(v,s(v,\theta)) = C\theta
	\]
	and can be chosen to be $C = c_{v,r}/\theta_r$ using reference values.
	Hence the Tait equation of state meets all requirements needed to have a strictly convex specific internal energy.
	
	It is worth noting, that in the barotropic case, i.e.\ when the pressure is a function of the specific volume alone, 
	the Tait equation of state and the stiffened gas equation of state are the same, cf.\ Thein \cite{hTHE}. 
	For further reading and details on the stiffened gas equation
	of state we refer to Fl{\aa}tten et al.\ \cite{hFMM} and the references therein.

\section{Applications of transformations that preserve convexity type properties}
\label{sec:appl}

We now present chains of transformations between thermodynamic functions that preserve convexity type properties,
including changes from convexity to concavity. These are applications of the transformation theorems in Section \ref{sec:convex}.
The convexity type properties can be non-strict, strict or the functions having
definite Hessian matrices.
With these chains of transformations
one can deduce these properties
from the simplest thermodynamic stability conditions, such as \eqref{eq:u_posdef}. If the internal energy has a desired
convexity property, then all other thermodynamic functions inherit the same convexity or concavity type.

\subsection{Relations between extensive functions and variables}
\label{subsec:ext_int}

Let us assume that the internal energy $U(V,S)$ is our fundamental function. We want to make use of the relations
\eqref{eq:theta_p} and \eqref{eq:S_deriv}. The fact that $U_S=\theta>0$ allows for the application of the
exchange of variables using Theorem \ref{thm:e} going from $U(V,S)$ to $S(V,U)$. Since $U_S>0$ convexity is switched
to concavity.

We may use the fact that by \eqref{eq:theta_p} $U_V=-p$ and $U_S=\theta$ to apply the Legendre transform \eqref{eq:legendre} 
to $U$. We set $\widehat{U}(-p,\vt)=\MI_L(U)(V,S)$. Analogously, we can use \eqref{eq:S_deriv} to define  
$\widehat{S}\left( \frac{p}{\vt}, \frac1{\vt}\right)=\MI_L(S)(V,U)$.
Further, we can set $G(p,\vt)=-\widehat{U}(-p,\vt)$ via an affine transformation $\mathcal{T}_A$ of type \eqref{eq:sign}.
The function $G$ is the concave Gibbs free energy or free enthalpy, see e.g.\ M\"uller and M\"uller \cite[(4.33)]{bMUMU}.
We summarize the results in the following diagram. The mapping going to the right is given above the arrow
and the inverse below
\begin{align*}
G(p,\vt) ~ \xLongleftrightarrow[-\mathcal{T}_{A}]{-\mathcal{T}_{A}^{-1}} ~
\widehat{U}(-p,\vt) ~ \xLongleftrightarrow[\mathcal{I}_{L}, ~ -pV+\vt S-U]{\mathcal{I}_{L}^{-1}} ~ 
U(V,S)  ~  \xLongleftrightarrow[\mathcal{I}_{E}^{-1}]{\,\mathcal{I}_{E},\;  U,S\,}  ~  
S(V,U)  ~ \xLongleftrightarrow[\mathcal{I}_{L}^{-1}]{\, \mathcal{I}_{L}, ~ (pV+U)/\vt -  S \,} ~
\widehat{S}\textstyle\left( \frac{p}{\vt}, \frac1{\vt}\right).
\end{align*}

\subsection{Relations between extensive and intensive variables}

We have for the specific internal energy $u(v,s)=\frac 1M U(Mv,Ms)=\frac 1M U(V,S)$. This
a simple linear rescaling, i.e.\ a special case of an affine transformation $\mathcal{T}_{A}$. Multiplication
by the constant mass $M$ does not change any convexity type properties.
We obtain the derivatives
$u_v=U_V$, $u_s=U_S$, $u_{vv}=M\cdot U_{VV}$, $u_{ss}=M\cdot U_{SS}$, $u_{vs}=M\cdot U_{VS}$.
The derivatives of $s$ are analogous.
Since $M>0$, the stability relations \eqref{eq:stab} and \eqref{eq:u_posdef} carry over to the specific state variables 
\begin{align*}
u(v,s) ~ \xLongleftrightarrow[\mathcal{T}_{A}]{\mathcal{T}_{A}^{-1}} ~ 
U(V,S)  ~  \xLongleftrightarrow[\mathcal{I}_{E}^{-1}]{\;\mathcal{I}_{E},\;  U,S\;}  ~  
S(V,U)  ~ \xLongleftrightarrow[\mathcal{T}_{A}^{-1}]{\mathcal{T}_{A}} ~ s(v,u).
\end{align*}
We also obtain equivalences to obtain the analog of first chain for the specific quantities
\begin{align*}
g(p,\vt) ~ \xLongleftrightarrow[\mathcal{T}_{A}]{\mathcal{T}_{A}^{-1}} ~
\widehat{u}(-p,\vt) ~ \xLongleftrightarrow[\;\mathcal{I}_{L}, ~ -pv+\vt s-u\;]{\mathcal{I}_{L}^{-1}} ~ 
u(v,s)  ~  \xLongleftrightarrow[\mathcal{I}_{E}^{-1}]{\;\mathcal{I}_{E},\;  u,s\;}  ~  
s(v,u)  ~ \xLongleftrightarrow[\mathcal{I}_{L}^{-1}]{\; \mathcal{I}_{L}, ~ (pv+u)/\vt -  s \;} ~
\widehat{s}\textstyle\left( \frac{p}{\vt}, \frac1{\vt}\right).
\end{align*}

The specific kinetic energy
energy $\frac{|\ul v|^2}2$ is strictly convex. It has the identity matrix in three dimension as positive definite Hessian matrix.
This implies that the specific total energy 
\begin{equation}
\label{eq:spec_tot_ener}
e(v,s,\ul v) =u(v,s)+\frac{|\ul v|^2}2,
\end{equation}
as the sum of specific internal and kinetic energies, inherits the convexity properties of the
specific internal energy $u$.  It has the Hessian matrix
with respect to $v$, $s$ and $\ul v$ given as
$$
\left(\begin{array}{ccccc}
u_{vv}&u_{vs}&0&0&0\\
u_{s v}&u_{s\eta}&0&0&0\\
0&0&1&0&0\\
0&0&0&1&0\\
0&0&0&0&1
\end{array}\right).
$$
For the specific total energy we have the following chain
\begin{align*}
s(v, e, \vv)~\xLongleftrightarrow[\;\mathcal{I}_{E},\;  u,s\;]{\mathcal{I}_{E}^{-1}}
~ e(v,s, \vv) ~ \xLongleftrightarrow[+\frac12 \abs{\vv}^2]{-\frac12 \abs{\vv}^2} ~ 
e(v,s)  ~  \xLongleftrightarrow[\mathcal{I}_{E}^{-1}]{\;\mathcal{I}_{E},\;  u,s\;}  ~  s(v,u).
\end{align*}

\subsection{The relation between specific state variables and their densities}

The reciprocal involution in Theorem \ref{thm:recinv} can be used to go back and forth between certain specific 
thermodynamic state variables
and their associated densities by using the fact that the 
specific volume $v$ is positive and therefore also the density $\rho =\frac 1v>0$. The possibility of
a vacuum state variable where $\rho =0$ is excluded in these considerations. 
If we apply \eqref{recip_fct} with $u_1=v$ and $w_1=\rho$ to any specific
state variable $\phi$ depending on other specific state variables, we obtain a density $\psi$ depending on densities as arguments.
For instance the specific total energy $e(v,s,\ul v)$ in specific variables given in \eqref{eq:spec_tot_ener} becomes the total
energy density
\begin{equation}
\label{eq:E_rho_S_M}
\ol E(\rho,\ol S,\ol{\ul M}) =\rho e\left(\textstyle\frac 1\rho,\frac {\ol S}\rho,\frac{\ol{\ul M}}\rho\right)
\displaystyle=\ol U(\rho,\ol S)+\frac{|\ol{\ul M}|^2}{2\rho}
\end{equation}
with densities as variables.
By Theorem \ref{thm:recinv} this function is again strictly convex. It has a positive definite Hessian matrix in the variables
$\rho$, $\ol S$ and $\ol{\ul M}$ if and only if $s$ has a negative definite Hessian matrix with respect to $v$ and $u$.
From the strict convexity of $u(v,s)$ we also obtain by Theorem \ref{thm:recinv} that the internal energy density $\ol U$ given as
\begin{equation}
\label{eq:U_rho_S}
\textstyle\ol U(\rho,\ol S)=\rho u\left(\frac 1\rho,\frac {\ol S}\rho\right)  
\end{equation}
is strictly convex in $\rho$ and $\ol S$. If the Hessian matrix of $u$ is positive definite then so is the Hessian matrix of $\ol U$.

The situation for the specific kinetic energy $e^{kin}(\ul v)=\frac{|\ul v|^2}2$ taken alone
is different. It is strictly convex in $\ul v$ 
and its Hessian matrix is the positive definite identity matrix $\ul I\in\R^{3\times 3}$. Since it does not depend on $v$
we cannot apply Theorem \ref{thm:recinv} as we did for the specific internal
energy $u$. The kinetic energy density can still be obtained by the reciprocal transformation using $v$ or by multiplication
of $e^{kin}$ by $\rho$ to give $\ol E^{kin}(\rho,\ol{\ul M})=\frac{|\ol{\ul M}|^2}{2\rho}$.
The derivatives are 
\begin{eqnarray*}
\ol E^{kin}_\rho &=&-\frac{|\ol{\ul M}|^2}{2\rho^2},\qquad \ol E^{kin}_{\ol{M}_i} =\frac{\ol{M}_i}\rho\quad\text{for}\; i=1,2,3,\\
\ol E^{kin}_{\rho\rho} &=&\frac{|\ol{\ul M}|^2}{\rho^3},\qquad \ol E^{kin}_{\ol{M}_i\rho} =\ol E^{kin}_{\rho \ol{M}_i} 
=-\frac{\ol{M}_i}{\rho^2}\qquad\text{and}\quad
\ol E^{kin}_{\ol{M}_i\ol{M}_j} =\frac{\delta_{ij}}\rho\quad\text{for}\; i,j=1,2,3.
\end{eqnarray*}
Applying \eqref{eq:conv_ineq} to $\ol E^{kin}$ gives 
\begin{eqnarray*}
\frac{|\ol{\ul M}_1|^2}{2\rho_1}-\frac{|\ol{\ul M}_2|^2}{2\rho_2}+\frac{|\ol{\ul M}_2|^2}{2\rho_2^2}(\rho_1-\rho_2) 
&&\hspace{-6mm}-\frac 1\rho_2\ol{\ul M}_2\cdot
(\ol{\ul M}_1-\ol{\ul M}_2) \\
= \frac 1{2\rho_1}&&\hspace{-6mm}\left(|\ol{\ul M}_1|^2+\frac{\rho_1^2|\ol{\ul M}_2|^2}{\rho_2^2}
-2\rho_1\frac{\ol{\ul M}_2\cdot\ol{\ul M}_1}{\rho_2}\right)\\
=\frac 1{2\rho_1}&&\hspace{-6mm}\left(\ol{\ul M}_1-\frac{\rho_1}{\rho_2}\ol{\ul M}_2\right)^2\ge 0.
\end{eqnarray*}
So $\ol E^{kin}$ is only convex and not strictly convex in $\rho$ and $\ol{\ul M}$. 
But adding it to the strictly convex internal energy density gives a strictly convex total energy density
$\ol E(\rho,\ol S,\ol{\ul M})$, as we have seen above.

The Hessian matrix of $\ol E^{kin}$ with respect to $\rho$ and $\ol{\ul M}$ is
\begin{eqnarray*}
\left(\begin{array}{cccc}
\frac{|\ol{\ul M}|^2}{\rho^3}&-\frac{\ol{M}_1}{\rho^2}&-\frac{\ol{M}_2}{\rho^2}&-\frac{\ol{M}_3}{\rho^2}\\
-\frac{\ol{M}_1}{\rho^2}&\frac 1\rho&0&0\\
-\frac{\ol{M}_2}{\rho^2}&0&\frac 1\rho&0\\
-\frac{\ol{M}_3}{\rho^2}&0&0&\frac 1\rho
\end{array}\right)
&&\\
=\left(\begin{array}{cccc}
1&-\frac{\ol{M}_1}\rho&-\frac{\ol{M}_2}\rho&-\frac{\ol{M}_3}\rho\\
0&1&0&0\\
0&0& 1&0\\
0&0&0& 1
\end{array}\right)
&&\hspace{-6.5mm}
\left(\begin{array}{cccc}
0&0&0&0\\
0&\frac1\rho&0&0\\
0&0&\frac 1\rho&0\\
0&0&0&\frac 1\rho
\end{array}\right)
\left(\begin{array}{cccc}
1&0&0&0\\
-\frac{\ol{M}_1}\rho&1&0&0\\
-\frac{\ol{M}_2}\rho&0& 1&0\\
-\frac{\ol{M}_3}\rho&0&0& 1
\end{array}\right).
\end{eqnarray*}
So it has rank 3 for $\rho>0$ and is only positive semi-definite, as expected.

Applying the reciprocal involution \eqref{eq:recip} and Theorem \ref{thm:recinv} to the specific entropy $s(v,u)$ gives
that the entropy density $\ol S$ given as
$$
\textstyle \ol S(\rho, \ol U) = \rho s\left(\frac 1\rho,\frac {\ol U}\rho\right)
$$
is concave in $\rho$ and $\ol U$ and has a negative definite Hessian matrix if $s$ in $v$ and $u$ does.

We obtain from \eqref{eq:E_rho_S_M} and \eqref{eq:U_rho_S} that
$$
\ol E_{\ol S}(\rho,\ol S,\ol{\ul M})=\ol U_{\ol S}(\rho,\ol S)=u_s\left(\textstyle\frac 1\rho,\frac {\ol S}\rho\right)=\theta(\rho,\ol S)>0.
$$
Therefore, we can apply the the change of variables in Theorem \ref{thm:e} to see that the entropy density
$\ol S$ in the conservative variables $\ul u=(\rho,\ol{\ul M},\ol E)^T$ has the concavity properties corresponding 
to the convexity properties
of total energy density $\ol E$ in the variables $\rho$, $\ol S$ and $\ol{\ul M}$. The concavity properties come 
essentially from those of the specific entropy 
$s$ as a function of $v$ and $u$ or the convexity properties of the specific internal energy $u$ as a function of $v$ and $s$.

There is an alternative route to proving the concavity properties of the entropy density $\ol S$ in the conserved variables $\ul u$.
We can also again consider the specific total energy as $e(v,\ul v,s)=u(v,s)+\frac{\ul v^2}2$ 
and use the fact $e_s=u_s=\theta >0$ in order to apply the interchange of a variable and 
a function in Theorem \ref{thm:e} to obtain $s(v,e,\ul v)$. Since $e_s>0$ the function
$s$ is strictly concave in the variables $v$, $\ul v$ and $e$ since $u$ is assumed to be strictly convex. 
Its Hessian matrix is negative definite if
that of $e$ is positive definite. Now the reciprocal involution using $v$ gives the entropy density as
\begin{equation}
\label{eq:entu}
\ol{S}(\ul u) =\ol{S}(\rho,\ol{\ul M},\ol E)
=\rho s\left(\textstyle\frac 1\rho ,\frac 1\rho\ol{\ul M},\frac {\ol E}\rho \right)=\rho s(\ul u)
\end{equation}
in the conserved variables $\ul u =(\rho,\ol{\ul M},\ol E)^T$. Again we obtain the concavity properties as above.
We see that the reciprocal transformation and the exchange of variables commute in this case.

We summarize these equivalences in the following scheme
\begin{align*}
s(v, u)& \quad \xLongleftrightarrow[\;\mathcal{I}_{E},\;  u,s\;]{\mathcal{I}_{E}^{-1}} \quad u(v,s) \qquad
\xLongleftrightarrow[-\frac12 \abs{\vv}^2]{\; +\frac12 \abs{\vv}^2\;} \qquad  e(v,s,\ul v) \quad
\xLongleftrightarrow[\mathcal{I}_{R}^{-1}]{\;\mathcal{I}_{R},\;\rho =\frac 1v\;}\quad\ol E(\rho,\ol{\ul M},\ol S)\\
\qquad\Updownarrow\mathcal{I}_{R},&\;\textstyle\rho =\frac 1v\displaystyle\qquad\qquad\;\;\Updownarrow\mathcal{I}_{R},\;
\textstyle\rho =\frac 1v\displaystyle
\qquad\qquad\qquad\quad\Updownarrow\mathcal{I}_{E},\;  e,s
\qquad\qquad\qquad\Updownarrow \mathcal{I}_{E},\; \ol E,\ol S\\
\ol S(\rho,\ol E)&\;\; \xLongleftrightarrow[\;\mathcal{I}_{E},\;  \ol U,\ol S\;]{\mathcal{I}_{E}^{-1}} \;\;\; \ol U(\rho,\ol S)
\qquad\qquad\qquad\qquad
s(v,e,\ul v)\quad\xLongleftrightarrow[\mathcal{I}_{R}^{-1}]{\;\mathcal{I}_{R},\;\rho =\frac 1v\;}\quad\ol S(\rho,\ol{\ul M},\ol E).
\end{align*}
If one of the functions in this scheme is strict, then so are all others. If one of them has a definite Hessian matrix,
then so do all others. Note that all the chains that we considered are linked via $u(v,s)$ and $s(v,u)$. 

\subsection{The generating potential in Godunov's theory}

In the theory of Godunov outlined in Section \ref{sec:cons} the convex entropy function 
$\Phi(\rho,\ol{\ul M},\ol E) =-\ol S(\rho,\ol{\ul M},\ol E)$ is used. Via the Legendre transformation \eqref{legendre} one obtains
convex generating potential $L(\up)$ in the main field variables $\up=\bs\Phi_{\ul u}(\ul u)$ with $\ul u =(\rho,\ol{\ul  M},\ol E)^T$.
We see that $L$ inherits the properties of $u(v,s)$ via appropriate transformations.

Godunov and Romenski \cite{bGORO} used the following interesting 
alternative chain to prove the convexity of $L$ in the main field variables.
They started with $e(v,s,\ul v)$ and used the Legendre transformation to obtain 
$\widehat{L}(-p,\theta,\ul v)=-pv+\theta s +\frac{|\ul v|^2}2-e(v,s,\ul v)$. Here the fact that $\frac{|\ul v|^2}2$ is invariant
under the Legendre transformation was used. Since $L_p=-v<0$, 
an exchange of variables $\MI_E$ of $p$ to $\widehat{L}$ was used to obtain
the convex function $p(-\widehat{L},\theta,\ul v)$. 
Next they used the fact that $\theta >0$ in order to apply an inverse reciprocal transformation
$\MI_R^{-1}$ using $\theta$. This gave the function 
$L(-\frac 1\theta\widehat{L},\frac 1\theta\ul v,\frac 1\theta) 
=\frac 1\theta p(-\frac 1\theta\widehat{L},\frac 1\theta\ul v,\frac 1\theta)$.
In his seminal paper Godunov \cite{hGOD1} had identified $-L$ as the concave generating potential that is the Legendre transform
of $\ol S(\rho,\ol {\ul M},\ol E)$ and the main field variables
in the form $-\frac 1\theta\widehat{L},-\frac 1\theta \ul v, \frac 1\theta$. 
We use $\up =(-\frac 1\theta\widehat{L},\frac 1\theta \ul v,-\frac 1\theta)$. Due to \eqref{eq:sign} the slight differences in
sign of the variables do not matter. We can summarize these results as follows
\begin{align*}
e(v,s) ~ \xLongleftrightarrow[-\frac12 \abs{\vv}^2]{+\frac12 \abs{\vv}^2} ~ 
e(v,s,\ul v) ~ \xLongleftrightarrow[\mathcal{I}_{L}^{-1}]{\;\mathcal{I}_{L},\;\widehat{L}\;} ~ &\widehat{L}(-p,\theta,\ul v)
~ \xLongleftrightarrow[\mathcal{I}_{E}^{-1}]{\;\mathcal{I}_{E},\;  p,\widehat{L}\;} ~ p(-\widehat{L},\theta,\ul v)
~ \xLongleftrightarrow[\mathcal{I}_{R}]{\;\mathcal{I}_{R}^{-1},\;\theta\;} ~ 
\textstyle L(-\frac 1\theta\widehat{L},\frac 1\theta\ul v,\frac 1\theta)\\
&\qquad\qquad\qquad\qquad\qquad\qquad\qquad\qquad\qquad\qquad\Updownarrow \text{using }\eqref{eq:sign}\\
& \qquad-\ol S(\rho,\ol{\ul M},\ol E)=\Phi(\rho,\ol{\ul M},\ol E) ~
\xLongleftrightarrow[\mathcal{I}_{L}^{-1}]{\;\mathcal{I}_{L},\;\up\;}~ 
\textstyle L(-\frac 1\theta\widehat{L},\frac 1\theta \ul v,-\frac 1\theta).
\end{align*}

Note that the first line of the above chain is the directest way to obtain a generating potential $L(\up)$ 
and main field variables, as discussed in Section \ref{sec:cons}, from the definition of a 
convex, strictly convex or positive convex specific internal energy of a thermodynamic system. 
One may add other terms that share the same convexity properties with the specific internal energy, 
such as the the specific momentum $\frac{|\ul v|^2}2$.
Then the generating potential inherits the same convexity properties.
Further, one can add terms to the generating potential $L(\up)$ that preserve its convexity properties.
This is a path to generate, starting from specific energies, symmetric hyperbolic systems 
of conservation laws for the conserved densities in the form \eqref{entvar-law2} 
by defining the fluxes as $L^i(\up) =v_iL(\up)$,
see e.g.\ Godunov and Romenski \cite[Chapter V and Appendix]{bGORO}.
As we see above, the formalism also leads to an entropy function $\bs\Phi$ sharing these properties and
an entropy density $\ol S=-\bs\Phi$ with analog concavity properties. One may thereby determine an entropy inequality that
is useful for identifying admissible discontinuous shock solutions to the system of conservation laws.

\section{The role of convexity in the analysis and numerics of compressible fluid flows}
\label{sec:role}

Thermodynamic stability and convexity of the energy as well as concavity of the entropy have an important impact 
on the analysis of compressible fluid flow models, such as the Euler and Navier-Stokes(-Fourier) systems. 
Here we present an overview of the theoretical results, including mathematical and numerical analysis.  

We first refer to the monograph of Majda \cite{bMAJ}, where the local-in-time existence of  the strong solution of the 
Euler system was established. Due to the concavity of the entropy, the Euler system can be transformed to 
a symmetric hyperbolic system and then the existence of a strong solution was derived. 

The convex structure of the energy is also heavily used in the limiting process to derive the global-in-time existence 
of weak or dissipative weak (measure-valued) solutions to the Navier-Stokes(-Fourier) system,
see Feireisl et al.\ \cite{bFEI, bFLMS, bFENO}. 
We should point out that the global-in-time existence of weak solutions is available only for specific pressure laws. 
In more general cases of pressure laws, the existence of dissipative (measure-valued) solutions was proven. 
The latter satisfy the underlying PDE system modulo defect measures arising in the momentum equation and the energy inequality. 
The convexity of energy gives rise to a positive sign of the energy defect, which in turn controls the Reynolds 
defect in the momentum equation.  

The global-in-time existence of dissipative (measure-valued) solutions can also be obtained 
by the convergence analysis of structure-preserving numerical methods, see Feireisl et al.\ \cite{hnFELU,hnFLM,bFLMS}. 
For the Euler system, the concavity of the entropy allows us to derive the weak bounded variation (BV) a priori estimates, 
with which the consistency of a numerical method can be established. 
Working in the framework of the generalised Lax equivalence principle, the convergence of a numerical scheme 
to a dissipative measure-valued solution was proven for stable and consistent numerical approximations in \cite{hnFLM}. 
We also refer to \cite{hnFELU,hnLSY1} where the same strategy was applied to the  Navier-Stokes(-Fourier) system.

The convexity of the energy density $\ol E$ with respect to $\ul u =(\vr, \ol{\ul M},\ol S)$
as well as the concavity of the entropy density $\ol S$ with respect to $(\vr, \ol{\ul M}, \ol E)$ are necessary to derive 
the relative energy/entropy as a Lyapunov functional in order to analyze stability properties. 
Indeed, the relative energy/entropy is a problem-suited ``metric", the so-called Bregman distance, 
that is used to measure the distance between two different solutions. 

Let $\ul u_1=(\vr_1,\vm_1,S_1)$ and $\ul u_2=(\vr_2,\vm_2,S_2)$ be two solutions of different systems. 
Then the relative energy $R_E$ can be viewed as the first-order Taylor expansion of $\ol E$ with respect to $\ul u$ defined by
\begin{align*}
\RE{\ul u_1}{\ul u_2} = \ol E(\ul u_1) - \ol E(\ul u_2)  - \ol E_{\ul u}(\ul u_2) \left(\ul u_1 - \ul u_2 \right).
\end{align*}
Due to the convexity of the energy $\ol E$, one knows that $\RE{\cdot}{\cdot} \geq 0$ and ``$=$" 
holds if and only if $\ul u_1 = \ul u_2$. 
Further, $L^p$-error can be derived from the relative energy due to the convexity of the thermodynamic variables. 
If $\ul u$ is uniformly bounded, then the equivalence between the relative energy and the $L^2$-errors holds. 
Hence, the relative energy is an elegant tool for analysis of the distance between two different solutions. 

We refer to Feireisl et al.\ \cite{hBRFE,bFLMS}. There the weak-strong uniqueness for weak or dissipative weak solutions 
was proven for the Euler system. In Feireisl and Novotn\'y \cite{bFENO} 
this property was shown for the Navier-Stokes(-Fourier) system. 
As weak solutions are generally not unique, the weak-strong uniqueness is crucial in order to derive stability 
of the strong solution in a larger class of generalised, weak or dissipative weak, solutions.
Consequently,  weak solutions turn out to be the strong solution as long as the latter exists.
This relative energy tool has also been shown to be successful in the numerical analysis of the error estimates. 
For the latter, we refer to Luk\'a\v cov\'a-Medvid'ov\'a et al.\ \cite{hnLSY} for the Euler system, \cite{hnBLMSY,hnLSY1} 
for the Navier-Stokes system, 
and  \cite{hnBEHL, hnLISH} for multi-component or multiphase models.
In addition, the relative energy functional provides elegant way to study the asymptotic behaviour of compressible fluid flows in 
their singular limits, e.g.\ as the Mach number vanishes, see Feireisl and Novotn\'y \cite{bFENO1}.\\[3mm]  

{\bf Acknowledgements:} The authors are grateful for financial support by the Sino-German cooperation group
{\sl Advanced Numerical Methods for Nonlinear Hyperbolic Balance Laws and
Their Applications} NSFC-DFG GZ 1465. 
The research of M.\ Lukacova was supported by the Deutsche Forschungsgemeinschaft  project
number 233630050-TRR 146 and through SPP 2410 (COSCARA), project number 525800857.  
She gratefully acknowledges  support of the Gutenberg Research College of the
University of Mainz and the Mainz Institute of Multiscale Modeling.
F.\ Thein gratefully acknowledges the Deutsche Forschungsgemeinschaft
(DFG, German Research Foundation) for the financial support through the project 525939417/SPP2410.
The work of Y.\ Yuhuan was supported by National Natural Science Foundation of China under grant No. 12401527 and 
Natural Science Foundation of Jiangsu Province under grant No. BK20241364.

\bibliography{c:/A-EIGENE_DATEIEN/Bibfiles/book,c:/A-EIGENE_DATEIEN/Bibfiles/hyp,c:/A-EIGENE_DATEIEN/Bibfiles/flu,%
c:/A-EIGENE_DATEIEN/Bibfiles/misc,c:/A-EIGENE_DATEIEN/Bibfiles/num,c:/A-EIGENE_DATEIEN/Bibfiles/trans,%
c:/A-EIGENE_DATEIEN/Bibfiles/publ,c:/A-EIGENE_DATEIEN/Bibfiles/hypnum} %
\bibliographystyle{plain}

\end{document}